\documentclass[a4paper, 11pt]{article}

\usepackage[titletoc,toc,title]{appendix}

\input xy   
\xyoption{all}
\usepackage{amsmath, amsthm, amsfonts, amssymb}
\usepackage{enumerate}
\usepackage{stackrel}
\usepackage{colordvi}
\usepackage{hyperref}
\usepackage[english]{babel}
\usepackage{color}
\usepackage{verbatim}
\usepackage[all]{xy}
\usepackage{graphicx}
\usepackage{mathrsfs}

\DeclareMathAlphabet{\mathcalligra}{T1}{calligra}{m}{n}
\DeclareMathAlphabet{\mathscrmin}{T1}{mathscr}{m}{n}

\numberwithin{equation}{section}
\theoremstyle{plain}
  \begingroup
        \newtheorem{theorem}[equation]{Theorem}
        
        \newtheorem{proposition}[equation]{Proposition}
        \newtheorem{corollary}[equation]{Corollary}

	    \newtheorem{definition}[equation]{Definition}
        \newtheorem{notation}[equation]{Notation}

\endgroup

\usepackage{amsmath,amssymb,amsthm, xspace,a4wide,color}
\theoremstyle{definition}
  \begingroup
        \newtheorem{remark}[equation]{Remark}
        \newtheorem{example}[equation]{Example}

\endgroup

\newcommand{\alg}[1]
{(#1,\overline{#1})}

\newcommand{\cc}{\mathcal}

\newcommand{\ff}{\mathsf}

\newcommand{\de}{definition}
\newcommand{\prop}{proposition}
\newcommand{\A}{\mathcal{A}}
\newcommand{\B}{\mathcal{B}}

\newcommand{\K}{\mathcal{K}}
\newcommand{\Cat}{\mathcal{C}at}
\newcommand{\eps}{\varepsilon}

\newcommand{\mr}[1]{\overset {#1} {\longrightarrow}}

\newcommand{\xr}[1]{\xrightarrow {#1}}

\newcommand{\Mr}[1]{\overset {#1} {\Longrightarrow}}

\newcommand{\mrpairviejo}[2]
   {
    \xymatrix@C=5ex@R=2.4ex
            {
             {} \ar@<1.6ex>[r]^{#1} 
	            \ar@<-1.1ex>[r]^{#2} 
	         & {}
            }
   }

\newcommand{\mrpair}[2]
   {
    \xymatrix@C=5ex@R=2.4ex
            {
             {} \ar@<1ex>[r]^{#1} 
	            \ar@<-1ex>[r]_{#2} 
	         & {}
            }
   }

 \newcommand{\mrpairc}[2]
   {
    \xymatrix@C=5ex@R=2.4ex
            {
             {} \ar@<1ex>[r]^{#1} 
	            \ar@<-1ex>[r]|{o}_{#2} 
	         & {}
            }
   }

 \newcommand{\mrpaircc}[2]
   {
    \xymatrix@C=5ex@R=2.4ex
            {
             {} \ar@<1ex>[r]|{o}^{#1} 
	            \ar@<-1ex>[r]|{o}_{#2} 
	         & {}
            }
   }

\newcommand{\mlpair}[2]
   {
    \xymatrix@C=5ex@R=2.4ex
            {
             {} 
              & {} \ar@<1.0ex>[l]_{#2} 
	          \ar@<-1.7ex>[l]_{#1}
            }
    }

\newcommand{\cellrd}[3] % flechas para adelante, cell para abajo %
 {
  \xymatrix@C=7ex@R=2.4ex
         {
          {} \ar@<1.6ex>[r]^{#1} 
             \ar@{}@<-1.3ex>[r]^{\!\! {#2} \, \!\Downarrow}
             \ar@<-1.1ex>[r]_{#3} 
          & {}
         }
 }
 \newcommand{\modif}[3] % cells para adelante, flecha para abajo %
 {
  \xymatrix@C=7ex@R=2.4ex
         {
          {} \ar@<1.6ex>@{=>}[r]^{#1} 
             \ar@{}@<-1.3ex>@{=>}[r]^{\!\! {#2} \, \!\downarrow}
             \ar@<-1.1ex>[r]_{#3} 
          & {}
         }
 }
 \newcommand{\scellrd}[3] % mas cortas flechas para adelante, cell para abajo %
 {
  \xymatrix@C=4.5ex@R=2.4ex
         {
          {} \ar@<1.6ex>[r]^{#1}
             \ar@{}@<-1.3ex>[r]^{\!\! {#2} \, \!\Downarrow}
             \ar@<-1.1ex>[r]_{#3}
          & {}
         }
}

\newcommand{\cellld}[3] % flechas para atras, cell para abajo %
 {
  \xymatrix@C=6ex@R=2.4ex
         {
            {} 
          & {} \ar@<1.0ex>[l]^{#3} 
          \ar@{}@<-1.7ex>[l]^{\!\! {#2} \, \!\Downarrow}
	                                 \ar@<-1.7ex>[l]_{#1}
         }
 }

\newcommand{\cellpairrd}[4] % como \cellrd pero con dos cells %
 {
  \xymatrix@C=10ex@R=2.4ex
         {
          {} \ar@<1.6ex>[r]^{#1} 
             \ar@{}@<-1.3ex>[r]^{\!\! {#2} \, \!\Downarrow 
                                 \;\; {#3} \, \!\Downarrow }
             \ar@<-1.1ex>[r]_{#4} 
          & {}
         }
 }

\newcommand{\cellpairrdcorto}[4] % como \cellrd pero con dos cells %
 {
  \xymatrix@C=6ex@R=2.4ex
         {
          {} \ar@<1.6ex>[r]^{#1} 
             \ar@{}@<-1.3ex>[r]^{\!\! {#2} \, \!\Downarrow 
                                 \;\; {#3} \, \!\Downarrow }
             \ar@<-1.1ex>[r]_{#4} 
          & {}
         }
 }

\newcommand{\cellpairrdc}[4] % como \cellrd pero con dos cells %
 {
  \xymatrix@C=10ex@R=2.4ex
         {
          {} \ar@<1.6ex>[r]^{#1} 
             \ar@{}@<-1.3ex>[r]^{\!\! {#2} \, \!\Downarrow 
                                 \;\; {#3} \, \!\Downarrow }
             \ar@<-1.1ex>[r]|{o}_{#4} 
          & {}
         }
 }
 
\newcommand{\cellpairrdcc}[4] % como \cellrd pero con dos cells %
 {
  \xymatrix@C=10ex@R=2.4ex
         {
          {} \ar@<1.6ex>[r]|{o}^{#1} 
             \ar@{}@<-1.3ex>[r]^{\!\! {#2} \, \!\Downarrow 
                                 \;\; {#3} \, \!\Downarrow }
             \ar@<-1.1ex>[r]|{o}_{#4} 
          & {}
         }
 }
 
   \newcommand{\soLim}[2]
   {
    \underset{#1}{\underleftarrow{\ff{\sigma \omega Lim}}}
    \; {#2}
   }

   \newcommand{\soopLim}[2]
   {
    \underset{#1}{\underleftarrow{\ff{\sigma \omega opLim}}}
    \; {#2}
   }

\newcommand{\dcell}[1]  % celda down %
          {
					 \ar@<8pt>@{-}[d]+<-4pt,8pt> 
           \ar@<-8pt>@{-}[d]+<4pt,8pt>
           \ar@{}[d]|{#1}
          }

\newcommand{\dcellb}[1]   % celda ancha down %
          {
           \ar@<10pt>@{-}[d]+<-5pt,8pt> 
           \ar@<-10pt>@{-}[d]+<5pt,8pt>
           \ar@{}[d]|{#1}
          }

\newcommand{\deq}        % identidad down %
         {
          \ar@{=}[d]
         }
        
\newcommand{\dreq}       % identidad down right %
         {
          \ar@{=}[dr]
         }

\newcommand{\dleq}       % identidad down left %
         {
          \ar@{=}[dl]
         }

\newcommand{\dccell}[1]    % celda de uno a dos lugares down % 
          {
           \ar@{-}[ld] 
           \ar@{-}[rd] 
           \ar@{}[d]|{#1}  
          }

\newcommand{\dcellbb}[1]   % celda mas ancha down %
          {
           \ar@<20pt>@{-}[d]+<-10pt,12pt> 
           \ar@<-20pt>@{-}[d]+<10pt,12pt>
           \ar@{}[d]|{#1}
          } 
\newcommand{\dl}    % una raya para poner a la izquierda (inclinada) % 
          {                        
           \ar@<-2pt>@{-}[d]+<4pt,8pt>
          }

\newcommand{\dr}    % una raya para poner a la derecha (inclinada) % 
          {                        
           \ar@<2pt>@{-}[d]+<-4pt,8pt> 
          }
\newcommand{\dc}[1]    % para poner el label en el centro % 
          {                        
           \ar@{}[d]|{#1}  
          }
\newcommand{\dcr}[1]    % para poner el label en el centro a la derecha para tamanos impares% 
          {                        
           \ar@{}[dr]|{#1}  
          }

\newcommand{\uccell}[1]      % celda de uno a dos lugares up %
          { 
           \ar@{-}[ur] 
           \ar@{}[u]|{#1} 
           \ar@{-}[ul] 
          }

\newcommand{\uccellb}[1]     % celda ancha de uno a dos lugares up %
          { 
           \ar@<-1ex>@{-}[ur] 
           \ar@{}[u]|{#1} 
           \ar@<1ex>@{-}[ul] 
          }

\newcommand{\dcellop}[1]  % celda down abriendo %
          {
					 \ar@<6pt>@{-}[d]+<6pt,8pt> 
           \ar@<-6pt>@{-}[d]+<-6pt,8pt>
           \ar@{}[d]|{#1}
          }

\newcommand{\dcellopb}[1]  % celda down abriendo mas ancha%
          {
					 \ar@<7pt>@{-}[d]+<7pt,8pt> 
           \ar@<-7pt>@{-}[d]+<-7pt,8pt>
           \ar@{}[d]|{#1}
          }

\newcommand{\did}{\ar@2{-}[d]}

\newcommand{\op}[1]
          {
           \ar@{-}[ld] 
           \ar@{-}[rd] 
           \ar@{}[d]|{#1}  
          }

\newcommand{\cl}[1]
          { 
           \ar@{-}[ur] 
           \ar@{}[u]|{#1} 
           \ar@{-}[ul] 
          }

\begin{document}

\title{A general limit lifting theorem for 2-dimensional monad theory}

\author{Szyld M.}

\maketitle

\begin{abstract}
We give a definition of weak morphism of $T$-algebras, for a $2$-monad $T$, with respect to an arbitrary family $\Omega$ of $2$-cells of the base $2$-category. By considering particular choices of $\Omega$, we recover the concepts of lax, pseudo and strict morphisms of $T$-algebras. We give a general notion of weak limit,
and define what it means for such a limit to be compatible with another family of $2$-cells.
These concepts allow us to prove a limit lifting theorem which unifies and generalizes three different previously known results of $2$-dimensional monad theory.
Explicitly, by considering the three choices of $\Omega$ above our theorem has as corollaries the lifting of oplax (resp. $\sigma$, which generalizes lax and pseudo, resp. strict) limits to the $2$-categories of lax (resp. pseudo, resp. strict) morphisms of $T$-algebras.
%. Various interesting results, some of which had independent and different proofs, follow as corollaries.
\end{abstract}

%\tableofcontents
 
\section{Introduction} \label{sec:intro}

The notion of a $2$-monad $T$ in an arbitrary base $2$-category $\K$ goes back to at least \cite{SMonad}, we refer the reader to the expository paper \cite[\S 3]{Review} for a detailed account of this theory in its early stages.
The relevance of 
$2$-dimensional monad theory, that is
the study of algebras of a $2$-monad and its morphisms, 
as a branch of $2$-dimensional universal algebra
is nowadays very well established.
The \emph{classic} article \cite{K2dim} is a standard reference, and we will use most of the notation as it is there.

It is well known that
the notion of strict algebra morphism coming from $\Cat$-enriched monad theory has to be relaxed in order to
go beyond this theory and into the real and deep $2$-dimensional one
(see \cite{K2dim}, \cite{Llax} and also \cite[3.5]{Review} for examples). 
In \cite[\S 1.2]{K2dim} lax and pseudomorphisms of $T$-algebras are introduced.
Instead of requiring the commutativity of a diagram, lax morphisms have a structural $2$-cell, which is required to be invertible for the pseudomorphism notion (which are referred to simply as morphisms due to its relevance for applications). 
Their choice of the notation $T$-$Alg$ for the $2$-category of pseudomorphisms of $T$-algebras is thus appropriate, but the notation $T$-$Alg_p$ is better suited for our point of view in this article, which is to consider all these notions of morphisms as particular cases of a general one. We do this in Section \ref{sec:weakmorphisms}, by considering an arbitrary family of $2$-cells where the structural $2$-cells are required to belong to.

The technique of fixing an arbitrary family of $2$-cells and considering weak morphisms and (conical) weak limits with respect to it is, as far as we know, due to Gray \cite{Gray}. A significant fact, observed by Street in \cite{S}, is that conical weak limits in this sense are as general as weighted strict limits (i.e. $2$-limits, as considered in \cite{KElem}), in the sense that any such limit can be expressed as one of the other type. Informally speaking, we may say that the weakening of the limit allows to work with only conical ones, without losing generality. 

 We give in Section \ref{sec:solimits} a notion of weighted weak limit, 
 appropriate for working with weak algebra morphisms,
 which has both notions above (that is, conical weak limits and weighted strict limits) as particular cases in a direct way. 
 If we interpret the result of \cite{S} in this context (we do this in \S \ref{sub:weightedconical}), we obtain that every weighted weak limit can be expressed as a conical weak limit.
 This result allows us to apply our Theorem \ref{teo:main}, which is a result for conical weak limits, to general weighted ones. 
 
We note that, since every weighted weak limit can be expressed as a conical weak limit, using \cite{S} again we have that ultimately it is a strict limit, but it is convenient to disregard this fact and work solely with the weak limit. 
 This is a situation that has previously happened in the literature with lax and pseudolimits: while ultimately being strict limits, these are the appropriate notions for many aspects of $2$-dimensional category theory and should not be considered as strict limits for these applications.

An important result for monad theory is the lifting of limits along the forgetful functor of the categories of algebras. 
To this subject we devote Sections \ref{sec:weightedlimitsofalgebras} and \ref{sec:weaklimitsofalgebras}.
The question of which limits can be lifted has been analysed for various cases, and it is also the subject of the recent article
\cite{LSAdv}, whose introduction we also suggest to the reader.

\begin{itemize}
 \item For the strict theory (or more generally for the $\cc{V}$-enriched theory, \cite{D}), it has been long known that all strict limits lift.

 \item For the $2$-category of pseudomorphisms, this is treated in \cite[\S 2]{K2dim}. Now it is no longer the case that all strict limits lift. First products, inserters and equifiers are lifted, and then this is used to lift other limits which are constructed from them, their main result being the lifting of pseudolimits, since these are the most important ones for its applications.
 
 \item Limits in the $2$-category of lax morphisms are considered of secondary interest in \mbox{\cite[Remark 2.9]{K2dim},} but an interest in them has justifiedly grown since then. In \cite{Llax} this subject is investigated, its most general theorem being the lifting of oplax limits. 

\end{itemize}

The article \cite{LSAdv} can be considered as a continuation of the results of \cite{K2dim} and \cite{Llax}, though in a direction quite different to ours.
%An analysis  deserves a separate paragraph. 
Their main result is a complete characterization (in terms of the weights) of those $2$-limits which lift to the $2$-category $T$-$Alg_\ell$ of lax morphisms of algebras for any $2$-monad $T$. They do so by introducing a particular base category $\cc{F}$, and working in \mbox{$\cc{F}$-enriched} category theory, in particular with $\cc{F}$-limits. 
This allows them to express in the language of enriched category theory properties about the distinguished $2$-subcategory of \mbox{$T$-$Alg_\ell$} consisting of the 
strict algebra morphisms. 
A motivation for these concepts lies in the lifting results of \cite[\S 4]{Llax}, 
where the hypothesis for the lifting of inserters and equalizers to \mbox{$T$-$Alg_\ell$}
 involve strict morphisms of algebras, these results are generalized to the \mbox{$\cc{F}$-enriched} case in \cite[3.5.3,3.5.4]{LSAdv}.
%and what gives \cite{LSAdv} its title is that this \emph{enhances} the purely 2-categorical part of the theory.

The main reason for the present paper's approach to this subject lies in my motivation, which was to obtain a proof, as simple as possible, of the lifting of $\sigma$-limits (see \cite{DDS1})
to the $T$-$Alg_p$ $2$-categories, and of the properties of such lifted limits. 
This result interested me particularly, as it allows for a different proof of \cite[Coro. 2.6.4]{DDS1}, which is a fundamental result for the theory of flat 2-functors that we develop there.
Though, with a careful checking, one could deduce this result from the results of \cite{K2dim}, I felt that this didn't lead to a satisfying answer to the question of \emph{why} such limits can be lifted.
%using monad theory, of the pointwise computation of $\sigma$-limits in the $Hom_p$ $2$-categories. 
%This fact is proved in \cite{DDS1}, and though it could also be deduced from the results of \cite{K2dim} (with a careful
There is an intuitive \emph{order} between the strictness of the different $2$-categories of morphisms of algebras, and I began thinking 
(inspired in part by our previous work that led to \cite[Prop. 2.6.2]{DDS1}) 
if one could think of such an order for the different types of limits. This is what led me to consider the general notions of $\omega$-$T$-morphism and $\sigma$-$\omega$-limit, which allowed me to formalize the intuition that, the more we relax the notion of morphism, the less limits we can lift.
Corollary \ref{coro:corogeneral} expresses this idea formally and clearly (relaxing the morphisms corresponds simply to augmenting the family $\Omega'$ of $2$-cells), and while it is not the most general result of the paper, it was the most satisfactory for me.
In particular, it provided me with an answer to my question (see Remark \ref{rem:casosk2dimysigmalim}):
$\sigma$-limits can be lifted to $T$-$Alg_p$ because they are $\sigma$-$\omega$-limits with respect to $\Omega_p$.

One fact that distinguishes the lax theory from the other two, as shown in the three items above, is the appearance of the ``op'' prefix. 
One cannot avoid to reverse the direction of some $2$-cell when working with lax morphisms; if instead of oplax one wishes to lift lax limits, then one must consider colax instead of lax morphisms (\cite[Theorem 4.10]{Llax}). 
In fact, it is convenient to think that the ``op'' is always present, but can be ommitted when restricting to morphisms with invertible structural $2$-cells (that is, pseudomorphisms and strict morphisms). As it is usually the case, the study of a more general theory yields light on some aspects of its particular cases as well; 
here in the proof of Theorem \ref{teo:main} regarding the lifting of weak op-limits, the necessity of reversing a direction is made transparent, see \eqref{eq:diagconop}.

Theorem \ref{teo:main} is our main contribution in this article, and it has the three main results of the items above as corollaries. 
Furthermore, unlike the one in \cite[Theorem 2.6]{K2dim}, its proof is direct, in the sense that it does not depend on the construction of the limit to be lifted in terms of other limits. 
Also, Propositions \ref{prop:products}, \ref{prop:inserters} and \ref{prop:equifiers} generalize the corresponding results of \cite{K2dim} and \cite{Llax} regarding the lifting of strict limits. A comment regarding its relevance can be found at the beginning of Section \ref{sec:weightedlimitsofalgebras}.

A fundamental notion which allows these generalizations is that of $\Omega$-compatibility (with respect to a family $\Omega$ of $2$-cells) for a limit,  see \S \ref{sub:compatible}. 
Informally, what it means for a limit to be $\Omega$-compatible is that the bijection at the level of $2$-cells given by its $2$-dimensional universal property restricts to $\Omega$. 
This notion holds automatically if $\Omega$ consists of the invertible \mbox{$2$-cells,} and this fact is implicitly used in the proofs of the first three propositions of \cite[\S 2]{K2dim}. By making it explicit, its generalizations in Section \ref{sec:weightedlimitsofalgebras} follow, with proofs that involve the same computations as the ones in \cite{K2dim}. We note that the proofs of the results in \cite{Llax} that we generalize here are substantially different to those of \cite{K2dim}, and thus of ours.
Also, our results have as corollaries the results of \cite{Llax} in the stronger form in which they appear in \cite[\S 6]{Llax}, requiring the algebra morphisms to be pseudo, not necessarily strict. 
Regarding \cite{LSAdv} again, 
%I should note that the proofs of the results in Section \ref{sec:weightedlimitsofalgebras} are closer to those of \cite[\S 2]{K2dim} than to those of \cite[\S 3,\S 4]{Llax}.  
%Noting that \cite{LSAdv} 
since that article 
restricts to the case of enhanced 2-categories of strict and lax morphisms (though the authors hope to address other cases in future papers), 
I believe thus our results of Section \ref{sec:weightedlimitsofalgebras} probably  
admit analogous \mbox{$\cc{F}$-enriched} versions, and this should 
be relevant for the case of enhanced 2-categories of pseudo and lax morphisms.

There is an extra property, present in the mentioned results of \cite{K2dim} and \cite{Llax}, of the projections of the lifted limits. Namely, they are strict, and they (jointly) detect strictness.
We also generalize this property under the light of $\Omega$-compatibility: we show as part of our results that the projections of the lifted limits are strict, and that they (jointly) detect $\Omega$-ness if the limit in the base category is $\Omega$-compatible. Since for the strict case (that is, when $\Omega$ consists only of the identities) every limit is $\Omega$-compatible, ours is indeed a generalization.

\section{Weak morphisms of algebras} \label{sec:weakmorphisms}

We fix an arbitrary family $\Omega$ of $2$-cells of a $2$-category $\cc{K}$, closed under horizontal and vertical composition, and containing all the  identity $2$-cells. 
Whenever we consider families of $2$-cells, we will assume these axioms to hold. 

We refer to \cite{K2dim} for the basic definitions of $2$-dimensional monad theory. 
We consider a $2$-monad $T = (T,m,i)$ on $\cc{K}$, and strict algebras of $T$.
In this section we will define the $2$-category of \emph{weak morphisms} of $T$-algebras (with respect to $\Omega$), which has as particular cases the $2$-categories 
of strict, pseudo and lax morphisms of $T$-algebras.

\begin{\de} \label{def:TalgW}
 A weak morphism, or $\omega$-morphism (with respect to $\Omega$), $(A,a) \xr{(f,\overline{f})} (B,b)$ between two $T$-algebras consists of an arrow $A \mr{f} B$ and a $2$-cell in $\Omega$, $$\xymatrix{TA \ar[r]^{Tf} \ar[d]_a \ar@{}[dr]|{\Downarrow \overline{f}} & TB \ar[d]^b \\
				  A \ar[r]_{f} & B}$$
 \noindent These data is subject to coherence conditions (with respect to $m,i$) that we consider safe to omit since the reader can find them in \cite[(1.2),(1.3)]{K2dim}, \cite[(3.17),(3.18)]{Review}, \cite[\S2]{Llax}.

 $2$-cells between $\omega$-morphisms $(f,\overline{f}) \Mr{\alpha} (g,\overline{g})$ are arbitrary (that is, not necessarily in $\Omega$) $2$-cells $f \Mr{\alpha} g$ required to satisfy the equation 
 
 \begin{equation} \label{eq:alg2cell}
  \vcenter{\xymatrix@R=1pc{ TA \ar@<2ex>[r]^{Tf} \ar@<-2ex>[r]_{Tg} \ar@{}[r]|{\Downarrow \; T\alpha} \ar[dd]_{a} & TB \ar[dd]^b \\ 
  \ar@{}[dr]|{\Downarrow \; \overline{g}} \\
		      A \ar[r]_g & B}}
  = 
  \vcenter{\xymatrix@R=1pc{TA \ar[r]^{Tf} \ar[dd]_a \ar@{}[dr]|{\Downarrow \; \overline{f}} & TB \ar[dd]^b \\ & \\
		      A  \ar@<2ex>[r]^{f} \ar@<-2ex>[r]_{g} \ar@{}[r]|{\Downarrow \; \alpha} & B}}
 \end{equation}

 In this way, with the evident laws of composition, we have a $2$-category $T$-$Alg_{\omega}^{\Omega}$ of \mbox{$T$-algebras} and $\omega$-morphisms, and a forgetful $2$-functor $T$-$Alg_{\omega}^{\Omega} \xr{U_{\omega}^{\Omega}} \cc{K}$. We refer to these $2$-categories as $2$-categories of weak $T$-algebra morphisms, or $\omega$-$T$-morphisms. Whenever possible, we will omit $\Omega$ from the notation and write $T$-$Alg_{\omega}$, $U_\omega$.
\end{\de}

\begin{remark} \label{rem:darvueltamorf}
We note (see also \cite[p.3]{K2dim}, \cite[Remark 2.1]{Llax}) that  reversing the direction of $\overline{f}$ yields a notion of $co \omega$-morphism $(f,\overline{f})$, we denote the $2$-category so defined by $T$-$Alg_{co \omega}$. Since a $co \omega$-morphism $(f,\overline{f})$ is just a $\omega$-morphism for the $2$-monad $T^{co}$ in $\K^{co}$, then we have $T$-$Alg_{co \omega}^{co} = T^{co}$-$Alg_{\omega}$.

Note that if $\Omega$ has only invertible $2$-cells, then taking the inverse of $\overline{f}$ yields an isomorphism of categories $T$-$Alg_{co \omega} \cong T$-$Alg_{\omega}$.
\end{remark}

\begin{example} \label{ex:Talgspl}
 By considering the families $\Omega_\gamma$, $\gamma=s,p,\ell$ consisting respectively of the identities, the invertible $2$-cells and all the $2$-cells of $\cc{K}$, we recover the $2$-categories $T$-$Alg_\gamma$ introduced in \cite{K2dim} (note that their $T$-$Alg$ is our $T$-$Alg_p$), that is:
 
 \begin{center}
 $T$-$Alg_{\omega}^{\Omega_s} = T$-$Alg_s$,  
 $\quad \quad T$-$Alg_{\omega}^{\Omega_p} = T$-$Alg_p$,
 $\quad \quad T$-$Alg_{\omega}^{\Omega_\ell} = T$-$Alg_\ell$.
 \end{center}

 In this way, all the known examples of such $2$-categories, found for example in \cite[\S6]{K2dim}, \mbox{\cite[\S1,\S5]{Llax},} are $2$-categories of $\omega$-$T$-morphisms. 
 \end{example}

 \section{$\sigma$-$\omega$-limits} \label{sec:solimits}
 
We fix throughout this section a family $\Sigma$ of arrows of a $2$-category $\A$, closed under composition and containing all the identities, and a family $\Omega$ of $2$-cells of a $2$-category $\B$. 
 
 In this section we will define the notion of weighted $\sigma$-$\omega$-limit (with respect to $\Sigma$ and $\Omega$), which is a generalization of weighted strict limits (\cite{KElem}) and of Gray's cartesian quasi-limits (\cite{Gray}). The notion of weighted strict limit can be recovered by considering a particular choice of $\Sigma$ and $\Omega$, and the notion of cartesian quasi-limit is exactly the notion of conical $\sigma$-$\omega$-limit, thus it corresponds to the particular weight $\triangle 1$ constant at the unit category.
 
 On the other hand, we show in Proposition \ref{prop:weightedcomoconical} that any weighted $\sigma$-$\omega$-limit can be expressed as a conical $\sigma$-$\omega$-limit (thus as a quasi-limit), and therefore by a result of Street (\cite[Theorem 14]{S}) also as a weighted strict limit. Though ultimately we are working only with strict limits, the notion of (weighted) $\sigma$-$\omega$-limit turned out to be the most convenient one for working with $2$-categories of weak morphisms.
 
 We begin by recalling from \cite{Gray} (while adapting the notation to one more adapted to the current literature, in particular to \cite{DDS1})
 the concepts of $\sigma$-$\omega$-natural transformation and conical $\sigma$-$\omega$-limit (in both cases with respect to $\Sigma$ and $\Omega$). The case in which $\Omega = \Omega_p$ consists of the invertible $2$-cells is considered with great detail in \cite{DDS1}, 
 where these concepts are referred to as $\sigma$-natural transformation and conical $\sigma$-limit.  
 
 \begin{\de} \label{de:sonatural}
  A $\sigma$-$\omega$-natural transformation between two $2$-functors $\A \mrpair{F}{G} \B$ is a lax natural transformation $\{FA \mr{\theta_A} GA\}_{A \in \A}$,  $\{Gf \theta_A \Mr{\theta_f} \theta_B Ff\}_{A\mr{f}B \in \A}$ such that $\theta_f$ is in $\Omega$ for each $f \in \Sigma$.
  
  A modification $\theta \mr{\rho} \theta'$ between 
  $\sigma$-$\omega$-natural transformations is the same as a modification between the underlying 
  lax natural transformations, that is a family $\{\theta_A \Mr{\rho_A} \theta'_A\}_{A \in \A}$ of $2$-cells of $\B$ such that for all $A \mr{f} B \in \A$, $\theta'_f \circ Gf \rho_A = \rho_B Ff \circ \theta_f$. 
  
  The $2$-category $Hom_{\sigma,\omega}^{\Sigma,\Omega}(\A,\B)$ has as objects the $2$-functors, as arrows the $\sigma$-$\omega$-natural transformations and as $2$-cells the modifications. Whenever possible, we will omit $\Sigma$ and $\Omega$ from the notation. We refer to these $2$-categories as $2$-categories of $\sigma$-$\omega$-natural transformations.
 \end{\de}

\begin{notation} \label{epsilonnotation}
 Consider the set $\cc{L}_{\A,\B}$ consisting of one label $(\sigma^\Sigma,\omega^\Omega)$ for each choice of $\Sigma$ and $\Omega$. 
There is a (partial) order in $\cc{L}_{\A,\B}$ defined by $(\sigma^\Sigma,\omega^\Omega) \leq (\sigma^{\Sigma'},\omega^{\Omega'})$ if and only if $\Sigma' \subseteq \Sigma$, $\Omega \subseteq \Omega'$. Note that in this case we have an inclusion
 $Hom_{\sigma,\omega}^{\Sigma,\Omega}(\A,\B) \stackrel{i}{\hookrightarrow} Hom_{\sigma,\omega}^{\Sigma',\Omega'}(\A,\B)$.
  
 Consider $\A_0$ the family of all the arrows of $\A$, and $\A_{id}$ consisting only of the identities. For $\gamma = s,p,\ell$, recall 
the families $\Omega_\gamma$ of Example \ref{ex:Talgspl}. 
  We have the labels $\gamma$ in the set $\cc{L}_{\A,\B}$ (making sense for any $\A,\B$): 
 
 $$s = (\sigma^{\A_0},\omega^{\Omega_s}), \quad p = (\sigma^{\A_0},\omega^{\Omega_p}), \quad \ell = (\sigma^{\A_{id}},\omega^{\Omega_\ell})$$ 
 
 Note that $s$ is the bottom element of $\cc{L}_{\A,\B}$, and $\ell$ is the top one.
 
For any label $\eps = (\sigma^\Sigma,\omega^\Omega) \in \cc{L}_{\A,\B}$ we denote $Hom_{\eps}(\A,\B) = Hom_{\sigma,\omega}^{\Sigma,\Omega}(\A,\B)$. In this way, for each choice of $\gamma$ as above, we recover the usual $2$-categories $Hom_{\gamma}(\A,\B)$ as $2$-categories of $\sigma$-$\omega$-natural transformations. 
\end{notation}

 \begin{\de} \label{de:conicalsolimit}
 Given a $2$-functor $\A \mr{F} \B$, we define the category of $\sigma$-$\omega$-cones with vertex $E$, $Cones_{\sigma,\omega}(E,F) = Hom_{\sigma,\omega}(\A,\B)(\triangle E,F)$, note that a $\sigma$-$\omega$-cone is a lax cone \mbox{$\{E \mr{\theta_A} FA\}_{A\in \cc{A}}$,}
\mbox{$\{\theta_{A} Ff \Mr{\theta_f} \theta_B \}_{A\mr{f} B \in \A}$} such that $\theta_f$ is in $\Omega$ for every $f$ in $\Sigma$. 
% The morphisms between two $\sigma$-cones correspond to their morphisms as lax cones. 

  The (conical) $\sigma$-$\omega$-limit 
  %(cartesian quasi-limit in \cite[]{Gray}) 
  of $F$ is the universal $\sigma$-$\omega$-cone, denoted \mbox{$\{\soLim{A\in \A}{FA} \mr{\pi_A} FA\}_{A\in \A}$,} \mbox{$\{\pi_{A} Ff \Mr{\pi_f} \pi_B \}_{A\mr{f} B \in \A}$} in the sense that for each \mbox{$B \in \B$,} post-composition with $\pi$ is an isomorphism of categories
\begin{equation} \label{eq:ssssconical}
\; \B(B,\soLim{A\in \A}{FA}) \mr{\pi_*} Cones_{\sigma,\omega}(B,F). 
\end{equation}
We refer to the arrows $\pi_A$, for $A \in \A$, as the projections of the limit.
 \end{\de}

 We should warn the reader that in \cite{DDS1} we chose to use a letter $w$ and denote $w$-$\sigma$-cone to indicate that the cone is taken with respect to a weight, in order to distinguish it from a (conical) $\sigma$-cone. For obvious reasons, it is convenient not to use such a notation here. Thus, in the present article, $\sigma$-$\omega$-cones can be either conical or weighted, with the presence of a weight indicating the latter.
 
 We will now generalize this notion to the notion of weighted $\sigma$-$\omega$-limit. We note that the category of cones that is needed in order to define weighted $\sigma$-$\omega$-limits can't be given (at least in an evident way) as a $\cc{H}om_{\sigma,\omega}$ category, as it is done for strict weighted limits and for $\sigma$-limits (in particular for weighted pseudo and lax limits as in \cite[\S5]{KElem}), since we do not have a family of $2$-cells of $\Cat$ to play the role of $\Omega$. What we do instead is to describe with detail the notion of lax cone (with respect to a weight), and this allows us to define which of the structural $2$-cells should be in $\Omega$ for it to be a $\sigma$-$\omega$-cone.
We remark that we could consider only the conical case and use \eqref{eq:isoBird} in order to apply this definition to the weighted case, but we prefer the more explicit approach of Definition \ref{de:NE} (and also of Definition \ref{def:compatible} for the notion of compatibility). From this point of view, Propositions \ref{prop:weightedcomoconical} and \ref{prop:weightedcomoconicalcompatible} confirm that our definitions are correct.
 
\begin{\de} \label{de:NE} 
Given $2$-functors $\A \mr{W} \Cat$, $\A \mr{F} \B$, and $E$ an object of $\B$, we denote 
$Cones_{\ell}^W(E,F) = \cc{H}om_{\ell}(\A,\Cat)(W,\B(E,F-))$. This is the category of lax cones (with respect to the weight $W$) 
% , and to the families $\Sigma, \Omega$ which are omitted from the notation) 
for $F$ with vertex $E$. 
Note that a lax cone $W \Mr{\theta} \B(E,F-)$ is given by the data \mbox{$\{WA \mr{\theta_A} \B(E,FA)\}_{A\in \A}$,} $\{(Ff)_* \theta_A \Mr{\theta_f} \theta_B Wf\}_{A\mr{f} B \in \A}$. This is given in turn by its components 
$$\left\{ 
E \cellrd{\theta_A(x)}{\theta_A(\varphi)}{\theta_A(y)} FA
%\xymatrix{E \ar@<2ex>[r]^{\theta_A(x)} 
%\ar@<-2ex>[r]_{\theta_A(y)} \ar@{}[r]|{\Downarrow 
%\theta_A(\varphi)} & FA}
\right\}_{x \mr{\varphi} y \in WA}, \quad \left\{ \vcenter{\xymatrix@R=1pc{ & FA \ar[rd]^{Ff} \\ E \ar[ru]^{\theta_A(x)} \ar[rr]_{\theta_B(Wf(x))}^{ (\theta_f)_x \; \Downarrow} && FB}}\right\}_{A\mr{f} B \in \A, x \in WA}$$

We define the category of $\sigma$-$\omega$-cones $Cones_{\sigma,\omega}^W(E,F)$ as the full subcategory of $Cones_{\ell}^W(E,F)$ consisting of those 
lax cones $\theta$ such that $(\theta_f)_x$ is in $\Omega$ for every $A\mr{f} B \in \Sigma$, and every $x \in WA$. 
For a $\sigma$-$\omega$-cone $\xi$ with vertex $E$, 
$\xymatrix@C=3ex{W \ar@{=>}[r]^>>>>{\xi} & \B(E,F-)}$,  
%\ar@{=>}[r]^>>>{\xi}$
we have a functor $\mu_B = \xi^*$ given by precomposition with $\xi$:

\vspace{-1ex}

\begin{equation} \label{eq:ssss}
 \B(B, E ) \mr{\mu_B} Cones_{\sigma,\omega}^W(B,F)  %\cc{H}om_{\sigma,\omega}(\A,\Cat)(W,\B(B,F-))
 \end{equation}
 
 \vspace{-2ex}
 
$$ B \cellrd{f}{\alpha}{g} E \quad \longmapsto \quad W \Mr{\xi} \B(E,F-) \cellrd{f^*}{\alpha^*}{g^*} \B(B,F-) $$

\noindent
The $\sigma$-$\omega$-limit of $F$ weighted by $W$, denoted $\{W,F\}_{\sigma,\omega}$ or more precisely $(\{W,F\}_{\sigma,\omega}, \xi)$, is a $\sigma$-$\omega$-cone $\xi$ 
with vertex $E = \{W,F\}_{\sigma,\omega}$, universal in the sense that 
$\mu_B = \xi^*$ in 
\eqref{eq:ssss} is an isomorphism.
%
%\begin{equation} \label{eq:ssss}
%\B(B, \{W,F\}_{\gamma} ) \mr{\xi^*} \cc{H}om_{\gamma}(\A,\Cat)(W,\B(B,F-))
% \end{equation} 
We refer to the arrows $\xi_A(x)$, for $A \in \A$ and $x \in WA$, as the projections of the limit.

As usual, an equivalent formulation of the universal property is that there is a representation $\mu_B$ natural in the variable $B$ (as in \eqref{eq:ssss}), 
%
% \begin{equation} \label{5.3kelly}
%  \B(B, \{W,F\}_{\gamma} ) \mr{\cong} \cc{H}om_{\gamma}(\A,\Cat)(W,\B(B,F-))
% \end{equation}
%
and $\xi$ is recovered setting $B = E$, $\xi = \mu_E(id_E)$.
\end{\de}

\begin{remark}\label{rem:oplimits}%[{cf. \cite[Remark 2.2.6]{DDS1}}]
By reversing the direction of the $2$-cells $\theta_f$ in Definition \ref{de:sonatural} we have the notion of $\sigma$-$\omega op$-natural transformation, which leads in turn to the definition of $\sigma$-$\omega op$-limit. The fact that every $\sigma$-$\omega$-limit in $\B$ is a $\sigma$-$\omega op$-limit in $\B^{co}$, and vice versa, goes back to at least \mbox{\cite[Proposition 1.5]{Bird}} where  the lax case is shown, see \cite[Remark 2.2.6]{DDS1} for details. There, for each weight $\A \mr{W} \Cat$, a $2$-functor $\A^{co} \mr{W^d} \Cat$ is constructed and we have \mbox{$\{W,F\}_{\sigma,\omega} = \{W^{d},F^{co}\}_{\sigma,\omega op}$.} 

Then, as it is usual in the literature for the lax case, we can state and prove general results for $\sigma$-$\omega$-limits, and use them for $\sigma$-$\omega op$-limits too. A similar fact is true for $\sigma$-$\omega$-colimits (that is, they are $\sigma$-$\omega$-limits in the dual $2$-category $\B^{op}$), though we don't explicitly consider them in this paper.
\end{remark}

\begin{remark}
 When $W$ is the functor $\triangle 1$, constant at the unit category, we recover Gray's notion of conical $\sigma$-$\omega$-limit. 
 To show the natural isomorphism between the categories of cones for each notion, consider the well-known isomorphism
 
 $$\cc{H}om_{\ell}(\A,\Cat)(\triangle 1, \B(E,F-)) \cong \cc{H}om_{\ell }(\A,\B)(\triangle E,F)$$
 
 which is given by the formulas, 
for $\theta$ in the left side and $\eta$ in the right side, $\eta_A=\theta_A(*)$ for $A\in \A$, $\eta_f=(\theta_f)_{*}$ for $A\mr{f} B \in \A$.
  Then this isomorphism restricts to an isomorphism
  
 $$Cones_{\sigma,\omega}^{\triangle 1}(E,F) \cong Cones_{\sigma,\omega}(E,F).$$ %\cc{H}om_{\sigma,\omega}(\A,\B)(\triangle E,F)$$
\end{remark}

\begin{example} \label{ex:varioslimits}
 We consider the families $\Omega_\gamma$, $\gamma=s,p,\ell$ of Example \ref{ex:Talgspl}, see also Notation \ref{epsilonnotation}. 
 
 \begin{enumerate}
  \item[1a.] If $\Omega = \Omega_\ell$, then $\{W,F\}_{\sigma,\omega}$ is the lax limit $\{W,F\}_{\ell}$ for any choice of $\Sigma$.
  \item[1b.] If $\Sigma = \A_{id}$, then $\{W,F\}_{\sigma,\omega}$ is the lax limit $\{W,F\}_{\ell}$ for any choice of $\Omega$.
  \item[2.] If $\Omega = \Omega_p$, then $\{W,F\}_{\sigma,\omega}$ is the $\sigma$-limit $\{W,F\}_\sigma$ of \cite[\S2]{DDS1}
  (since for each natural transformation $\theta_f$ of the structure of a lax cone, $\theta_f$ is an isomorphism if and only if each $(\theta_f)_x$ is)
  . In particular, when $\Sigma = \A_0$ then $\{W,F\}_{\sigma,\omega}$ is the pseudolimit $\{W,F\}_{p}$. 
  \item[3.] If $\Omega = \Omega_s$, and $\Sigma = \A_0$, then $\{W,F\}_{\sigma,\omega}$ is the strict limit $\{W,F\}_{s}$.
 \end{enumerate} 
\end{example}

\subsection{$\Omega$-modifications and $\Omega$-compatible limits} \label{sub:compatible}

We will now define the notion of $\Omega'$-compatible $\sigma$-$\omega$-limit, for $\Omega'$ another family of $2$-cells of $\B$. 
We begin with the conical case, for which it can be defined using the concept of $\Omega'$-modification, and consider then the general weighted case. 
%We note that this notion, and the rest of this section, are not needed for our results in section \ref{sec:weightedlimitsofalgebras}.
%We include them here because we consider this is the correct location for these results, but the reader, if so inclined, could jump to section \ref{sec:weightedlimitsofalgebras} and come back here before starting with section \ref{sec:weaklimitsofalgebras}.
  
 \begin{\de}
 Consider a family $\Omega'$ of $2$-cells of $\B$. A $\Omega'$-modification $\theta \mr{\rho} \theta'$ between two lax natural transformations is a modification $\{\theta_A \Mr{\rho_A} \theta'_A\}_{A \in \A}$ such that $\rho_A$ is in $\Omega'$ for every $A \in \A$.
 
 We denote by $Hom_{\sigma,\omega}^{\Sigma,\Omega}(\A,\B)^{\Omega'}$ the $2$-category of $2$-functors, $\sigma$-$\omega$-natural transformations (with respect to $\Sigma, \Omega$) and $\Omega'$-modifications.
 
 Also, given objects $B,C$ of $\B$, we denote by $\B^{\Omega'}(B,C)$ the category with objects $B \mr{f} C$ and arrows $f \Mr{\alpha} g$ such that $\alpha \in \Omega'$.
 \end{\de}

  \begin{\de} \label{def:conicalcompatible}
  Let $\cc{A} \mr{F} \cc{B}$, we denote \mbox{$Cones_{\sigma,\omega}(E,F)^{\Omega'} := Hom_{\sigma,\omega}(\A,\B)(\triangle E,F)^{\Omega'}$}
(recall Definition \ref{de:conicalsolimit}).  We say that the limit $L = \soLim{A\in \A}{FA}$ is compatible with $\Omega'$, or \mbox{$\Omega'$-compatible,} if for every $B \in \B$, the restriction of the isomorphism in \eqref{eq:ssssconical} to \mbox{$\B^{\Omega'}(B,L) \mr{\pi_*} Cones_{\sigma,\omega}(B,F)^{\Omega'}$}
  is still an isomorphism.
%   restricts to $\Omega$, in the sense that for every $2$-cell $E \cellrd{f}{\gamma}{g} L$, $ $
 \end{\de}

\begin{remark} \label{rem:significadocompatible}
Consider the isomorphism $\pi_*$ in \eqref{eq:ssssconical} and let $\pi_*\varphi = \theta$, $\pi_*\varphi' = \theta'$. Then, for each morphism of cones given by $\theta_A \Mr{\alpha_A} \theta_A'$, there is a unique $2$-cell $\varphi \Mr{\beta} \varphi'$ such that $\pi_A \beta = \alpha_A$. The $\Omega'$-compatibility of the limit means that, if all the components $\alpha_A$ are in $\Omega'$, then so is $\beta$ (with the other implication always holding by the equation above).
\end{remark} 
 
\begin{remark} \label{rem:compatiblegratis}
  Consider as $\Omega'$ the families $\Omega_\gamma$, $\gamma=s,p,\ell$ of Example \ref{ex:Talgspl}. Then every (conical) $\sigma$-$\omega$-limit is $\Omega'$-compatible. In fact, for $\gamma=\ell$ the condition is vacuous, and for $\gamma=p,s$ this follows from the fact that (with the notation of the previous remark) 
invertible (resp. identity) $\alpha$ and $\beta$ correspond via the isomorphism $\pi_*$ (note that 
  a modification $\alpha$ is invertible, resp. the identity if and only if each of its components $\alpha_A$ are so).
 \end{remark}

 \begin{\de} \label{def:compatible}
  Let $\cc{A} \mr{F} \cc{B}$, $\cc{A} \mr{W} \Cat$, recall from Definition \ref{de:NE} the notion of \mbox{$\sigma$-$\omega$-cone} with respect to $W$. Note that a morphism of $\sigma$-$\omega$-cones (i.e. of the underlying lax cones) is a modification given by a family of natural transformations $\theta_A \Mr{\rho_A} \theta'_A$ and thus by a family of \mbox{$2$-cells}
  $E \cellrd{\theta_A(x)}{\rho_A(x)}{\theta'_A(x)} FA $ of $\B$, one for each $A \in \A$ and each $x \in WA$. We define $Cones_{\sigma,\omega}^W(E,F)^{\Omega'}$ to be the category with the same objects as $Cones_{\sigma,\omega}^W(E,F)$, that is \mbox{$\sigma$-$\omega$-cones,} and arrows those morphisms $\rho$ such that $\rho_A(x)$ is in $\Omega'$ for each $A \in \A$ and each $x \in WA$.
  
  We say that the limit $E = \{W,F\}_{\sigma,\omega}$ is compatible with $\Omega'$, or $\Omega'$-compatible, if for every $B \in \B$, the restriction of the isomorphism in \eqref{eq:ssss} to $\B^{\Omega'}(B,E) \mr{\xi^*} Cones_{\sigma,\omega}^W(B,F)^{\Omega'}$
  is still an isomorphism.
%   restricts to $\Omega$, in the sense that for every $2$-cell $E \cellrd{f}{\gamma}{g} L$, $ $
 \end{\de}

We leave to the reader the easy task of checking that, in the case $W = \triangle 1$, both definitions coincide (see also Proposition \ref{prop:weightedcomoconicalcompatible} below).

 We now examine carefully what it means to be $\Omega'$-compatible for some limits with which we will work in Section \ref{sec:weightedlimitsofalgebras}. We refer to \cite[\S 4]{KElem} for details.
 
 \begin{example} \label{ex:inserter}
  Recall that the inserter of a parallel pair of arrows $B \mrpair{f}{g} C$ is the (strict) limit of the diagram $\{a \mrpair{u}{v} b\} \mr{F} \B$ defined in an evident way, weighted by \mbox{$\{a \mrpair{u}{v} b\} \mapsto \{1 \mrpair{0}{1} 2\}$,} where we denote $1=\{*\}$, $2=\{0 \mr{\varphi} 1\}$. A cone for this diagram, $W \mr{\theta} \B(E,F-)$ has components $\theta_a, \theta_b$ but if we denote $q = \theta_a(*)$, since $\theta_b(0)=fq$ and $\theta_b(1)=gq$ then $\theta$ is determined by $q$ and $\mu=\theta_b(\varphi): fq \mr{} gq$.  
   In a similar way, a morphism of cones, i.e. a modification $\theta \mr{\rho} \theta'$, is determined by the $2$-cell $\beta = (\rho_a)_*: q \Rightarrow q'$, since $(\rho_b)_0 = f\beta$ and $(\rho_b)_1 = g\beta$, and the requirement that $\rho_b$ is a natural transformation is the equation $\mu'(f\beta) = (g\beta)\mu$ that $\beta$ must satisfy in order to determine a modification. 
  
  From this description of the category of cones, it follows the description of the universal property of the inserter as it is found in \cite[\S4]{KElem}: it is the universal pair $(L \mr{p} B,fp \Mr{\lambda} gp)$. 
We abuse language and refer to $p$ as \emph{the} projection of the inserter, though according to Definition \ref{de:NE} so are $\xi_b(0)$ and $\xi_b(1)$. 
  The one-dimensional universal property is that for any $(q,\mu)$ as above there is a unique $E \mr{h} L$ such that $ph=q$ and $\lambda h = \mu$. The two-dimensional universal property is that for $E \mrpair{h}{k} L$ and $ph \Mr{\beta} pk$ such that $(\lambda k)(f\beta) = (g\beta)(\lambda h)$, there is a unique $2$-cell $h \Mr{\alpha} k$ with $p\alpha = \beta$. Now, from Definition \ref{def:compatible} and the computations above it is clear that the inserter is $\Omega'$-compatible if, when $\beta$ is in $\Omega'$, so is $\alpha$ (note that if $\beta=(\rho_a)_*$ is in $\Omega'$, then by the computations above so are $(\rho_b)_0$ and $(\rho_b)_1$).
 \end{example}

 \begin{example} \label{ex:equifier}
We consider also the equifier of a pair of $2$-cells $B \cellpairrd{f}{\alpha}{\beta}{g} C$, where now the weight is given by the diagram (in $\Cat$) $1 \cellpairrdcorto{0}{}{}{1} 2$.
A similar process to the one of Example \ref{ex:inserter} shows that cones are now determined by arrows $E \mr{q} B$ such that $\alpha q = \beta q$, and morphisms by $2$-cells $q \Mr{\mu} q'$.

From here it follows that the universal property of the equifier $L \mr{p} B$ is (again, as in \cite[\S4]{KElem}) that given any other $E \mr{q} B$ as above there is a unique unique $E \mr{h} L$ such that $ph=q$; and that for $E \mrpair{h}{k} L$ and $ph \Mr{\mu} pk$, there is a unique $2$-cell $h \Mr{\lambda} k$ with $p\lambda = \mu$.
We refer to $p$ as \emph{the} projection of the equifier. Now, from the above it is clear that the equifier is $\Omega'$-compatible if, when $\mu$ is in $\Omega'$, so is $\lambda$.
 \end{example}

\subsection{The expression of weighted limits as conical limits} \label{sub:weightedconical}

Consider $2$-functors $\A \mr{W} \Cat$, $\A \mr{F} \B$, $\A \mr{G} \Cat$ and the $2$-category of elements $\cc{E}l_W$ of $W$. There is a well-known isomorphism (see \cite[Proposition 1.14]{Bird} for a proof)

\begin{equation} \label{eq:isoBird114}
Hom_{\ell}(\A,\Cat)(W,G) \mr{\cong} Hom_{\ell}(\cc{E}l_W,\Cat)(\triangle 1,G\Diamond_W)
\end{equation}

As noted in \cite{Bird}, this isomorphism implies immediately that any weighted lax limit can be expressed as a conical lax limit, but it also has the following interesting corollaries when restricting to ``stricter" limits. 
In \cite[Theorem 15]{S} Street uses this to show that any strict weighted limit can be expressed a 
conical $\sigma$-$\omega$-limit with $\Omega = \Omega_s$ (then in particular as a Gray's quasi-limit). In \cite[\S2.3]{DDS1} this procedure is slightly modified in order to prove that weighted $\sigma$-limits 
(given a family $\Sigma$ of arrows of $\A$)
can be expressed as conical $\sigma$-limits. 

The notion of $\sigma$-$\omega$-limit allows to easily show that these three results are instances of the same following proposition  
(considering respectively the items $1a,3,2$ of Example \ref{ex:varioslimits}), stating that any weighted $\sigma$-$\omega$-limit can be expressed a 
conical $\sigma$-$\omega$-limit with respect to the same family of $2$-cells $\Omega$.

% Consider 

The indexing category for the conical limit is, as in the particular cases described above, the $2$-category of elements $\cc{E}l_W$ of the weight. 
The diagram for the conical limit is given by the composition of the original $2$-functor $F$ with the projection $\cc{E}l_W \mr{\Diamond_W} \A$.  
The distinguished arrows of $\cc{E}l_W$ are those of the form $(f,id)$, with $f \in \Sigma$. We denote this family of arrows by $id_\Sigma$. We note that in \cite{DDS1} the family consisting of the arrows of the form $(f,\varphi)$, with $f \in \Sigma$ and $\varphi$ an isomorphism is considered instead; below the proof we explain why this family works for the $\sigma$-case (that is with $\Omega = \Omega_p$) but not for this general case.

% a $\Cat$-valued $2$-functor.
% We first need to recall the $2$-category of elements $\cc{E}l_P$ of a $\Cat$-valued $2$-functor.

\begin{\prop} \label{prop:weightedcomoconical}
Let $\A \mr{W} \Cat$, $\A \mr{F} \B$.
 The weighted $\sigma$-$\omega$-limit $\{W,F\}_{\sigma,\omega}$ is equal to the conical $\sigma$-$\omega$-limit of the $2$-functor $\cc{E}l_W \mr{\Diamond_W} \A \mr{F} \Cat$ (with respect to $id_{\Sigma}$ and $\Omega$), in the sense that the universal properties defining each limit are equivalent.
\end{\prop}

\begin{proof}
For any other $\A \mr{G} \Cat$, we recall that the isomorphism \eqref{eq:isoBird114} is given at the level of objects by the formulas, for $\eta$ on the left side and $\theta$ on the right side, $\eta_A(x) = \theta_{(x,A)}$, \mbox{$(\eta_f)_x = \theta_{(f,id)}$,} $\theta_{(f,\varphi)} = \eta_B(\varphi) (\eta_f)_x$. 
In particular, for $G = \B(E,F-)$ we consider the isomorphism between the categories of lax cones
 
\begin{equation} \label{eq:isoBird} 
 Cones_{\ell}^W(E,F) \cong Cones_{\ell}(E,F\Diamond_W).
 \end{equation}
 
 By the formula $(\eta_f)_x = \theta_{(f,id)}$ above, each $(\eta_f)_x$ is in $\Omega$ if and only if each $\theta_{(f,id)}$ is, therefore 
 if we consider the family $id_{\Sigma}$ of $\cc{E}l_W$, the isomorphism above restricts to
 
 $$Cones_{\sigma,\omega}^W(E,F) \cong Cones_{\sigma,\omega}(E,F\Diamond_W),$$
 
 thus finishing the proof.  
\end{proof}

We note that, if $f \in \Sigma$ and $\varphi$ is an isomorphism, from the formula $\theta_{(f,\varphi)} = \eta_B(\varphi) (\eta_f)_x$ it follows that $\theta_{(f,\varphi)} \in \Omega_p$ (in other words that it is invertible) if $(\eta_f)_x$ is, but this may not be the case for an arbitrary $\Omega$ (which may not contain $\Omega_p$).

\begin{remark} \label{rem:correspconos}
In the hypothesis of Proposition \ref{prop:weightedcomoconical}, the correspondence between the limit cones is given by the formulas present at the beginning of the proof of the proposition. If we denote these cones by $\xi$ and $\pi$ as in Definitions \ref{de:conicalsolimit} and \ref{de:NE}, we have in particular the formula $\xi_A(x) = \pi_{(x,A)}$ for each $A \in \A$, $x \in WA$, showing that the projections of each limit correspond.
\end{remark}

\begin{proposition} \label{prop:weightedcomoconicalcompatible}
In the hypothesis of Proposition \ref{prop:weightedcomoconical}, let $\Omega'$ be another family of $2$-cells of $\B$. Then $\{W,F\}_{\sigma,\omega}$ is $\Omega'$-compatible if and only if its conical expression $\soLim{}{F\Diamond_W}$ is so.
\end{proposition}

\begin{proof}
Recall that the isomorphism \eqref{eq:isoBird114} is given at the level of arrows by the formula, for $\eta \mr{\alpha} \eta'$ on the left side and $\theta \mr{\beta} \theta'$ on the right side, $\alpha_A(x) = \beta_{(x,A)}$ (for each $A \in \A$, $x \in WA$). Then, by Definitions \ref{def:conicalcompatible} and \ref{def:compatible}, the last isomorphism of the proof of Proposition \ref{prop:weightedcomoconical} restricts to
 $$Cones_{\sigma,\omega}^W(E,F)^{\Omega'} \cong Cones_{\sigma,\omega}(E,F\Diamond_W)^{\Omega'},$$
 yielding the desired result.
\end{proof} 
 
\begin{example} \label{ex:inserter2}
We examine the result of Proposition \ref{prop:weightedcomoconical} when applied to the inserter (see Example \ref{ex:inserter}), as this will be relevant for our results (see Example \ref{ex:inserterfinal}). 

In this case the indexing category $\cc{E}l_W$ of the conical $\sigma$-$\omega$-limit has three objects $(*,a)$, $(0,b)$, $(1,b)$ and three arrows 
from which all arrows can be obtained as identities and compositions, namely $(u,id)$, $(v,id)$, $(id,\varphi)$ which are mapped respectively to $f$, $g$, $id$.

This ``conical inserter'' is different to the following 
one 
that can be found in \cite[I,7.10 2)]{Gray}: there Gray considers the diagram $\{a \mrpair{u}{v} b\} \mapsto \{B \mrpair{f}{g} C\}$, $\Sigma = \{v,id_a,id_b\}$ and $\Omega = \Omega_s$.

In particular, we note that in the first case there are arrows of the family $id_\Sigma$ that are mapped to both $f$ and $g$, but in the second one no arrow of $\Sigma$ is mapped to $f$.
\end{example}

\section{Strict limits in the $2$-categories of weak morphisms} \label{sec:weightedlimitsofalgebras}
 
In this section we generalize the first limit lifting results of \cite[\S 2]{K2dim} to the $2$-categories of weak $T$-algebra morphisms, showing that these have products, inserters and equalizers when the base $2$-category does (under some extra conditions, which are trivially satisfied in the cases for which these results were previously known). Their case is recovered setting $\Omega = \Omega_p$. While these results are used in the rest of \cite[\S 2]{K2dim} to deduce the lifting of lax and pseudolimits (using the construction of lax and pseudolimits from those three types of limits), we decided not to attempt a generalization of this deduction, since we will give in Section \ref{sec:weaklimitsofalgebras} a different, direct proof of a theorem regarding the lifting of $\sigma$-$\omega$-limits, therefore in particular of lax and pseudolimits (see Remark \ref{rem:casosk2dimysigmalim}).

%On the one hand
Though we don't use them for our main result of Section \ref{sec:weaklimitsofalgebras}, the results of this section have their own relevance. First, setting $\Omega = \Omega_\ell$ we have as a particular case of our Propositions \ref{prop:inserters} and \ref{prop:equifiers} the results 
\cite[Prop. 4.3, 4.4]{Llax}, in the strengthened form of \cite[\S6]{Llax}. We note that the proofs of these results, as they appear in \cite{Llax}, are substantially different from the ones of \mbox{\cite[Prop. 2.2,2.3]{K2dim}}, thus also to the ones here. Also, though our lifting theorem of Section \ref{sec:weaklimitsofalgebras} can also be applied to products, inserters and equifiers, when doing so we obtain in some cases a result which requires stronger hypotheses than the ones in this section, we examine this with detail for the inserter in Example \ref{ex:inserterfinal}.
 
% 
% By considering the three choices of $\Omega$ in example \ref{ex:Talgspl}, and $\Omega' = \Omega_s$,  
% Note that by remark \ref{rem:compatiblegratis} the hypothesis of $\Omega$-compatibility and $\Omega'$-compatibility are immediately satisfied in these cases.
% 
\begin{definition} \label{de:preserve}
 Let $\Omega$, $\Omega'$ be any two families of $2$-cells of $\cc{K}$. We say that a family of morphisms $L \mr{p_i} A_i$ in $T$-$Alg_{\omega}^{\Omega}$ (jointly) detects $\Omega'$-ness if, for any other morphism $Z \mr{z} L$ in $T$-$Alg_{\omega}^{\Omega}$, if all the compositions $p_i z$ are $\omega$-morphisms with respect to $\Omega'$, then so is $z$. 
 
 If $\Omega' = \Omega_s$, we say "detect strictness". If 
$\Omega' = \Omega_p$, we say "detect pseudoness".\end{definition}

  \begin{\prop} \label{prop:products}
  The forgetful $2$-functor $T$-$Alg_{\omega}^{\Omega} \mr{U_{\omega}^{\Omega}} \cc{K}$ creates products of $T$-algebras for which their product in $\K$ is $\Omega$-compatible. The projections of the product are strict, and jointly detect $\Omega'$-ness for any family $\Omega'$ such that the product in $\K$ is also $\Omega'$-compatible. 
   \end{\prop}

\begin{proof}
  Our proof follows the same lines of \cite[Prop. 2.1]{K2dim}. Since we also use the same notation we decided to omit some calculations that can be explicitly found in op.cit. 
 If $A = \prod_{i \in I} A_i$ is an $\Omega$-compatible product in $\K$ of a family of $T$-algebras, we have $TA \mr{a} A$ the unique map with $p_i a = a_i \cdot Tp$. Then, as in \mbox{\cite[Prop. 2.1]{K2dim},} $(A,a) \in T$-$Alg_\omega^{\Omega}$ and $p$ is strict.
  
  Now, to show that $(p_i,id)$ is a product in $T$-$Alg_\omega^\Omega$, let $D \xr{(q_i,\overline{q_i})} A_i$ be $\omega$-morphisms. By the one-dimensional universal property in $\cc{K}$, there is a unique $D \mr{h} A$ with $p_i h = q_i$. Now, since the product is $\Omega$-compatible (see Remark \ref{rem:significadocompatible}), given the $2$-cells $p_i a . Th \Mr{\overline{q_i}} p_i hd$ there
is a unique 2-cell $a. Th \Mr{\overline{h}} hd$ in $\Omega$ with $p_i \overline{h} = \overline{q_i}$. 
   Note that, if the product is $\Omega'$-compatible, $\overline{h}$ is in $\Omega'$ if all the $\overline{q_i}$ are, giving the last assertion of the proposition. The rest of the proof (that is, the verification of the algebra axioms for $(h,\overline{h})$ and the two-dimensional universal property) is exactly as in \cite{K2dim}. 
\end{proof}

 \begin{\prop} \label{prop:inserters}
  The forgetful $2$-functor $T$-$Alg_{\omega}^{\Omega} \mr{U_{\omega}^{\Omega}} \cc{K}$ creates inserters of pairs $(f,\overline{f}), (g,\overline{g})$ for which $\overline{f}$ is invertible and the inserter of $(f,g)$ in $\K$ is $\Omega$-compatible. The projection of the inserter is strict, and detects $\Omega'$-ness for any family $\Omega'$ such that the inserter of $(f,g)$ in $\K$ is also $\Omega'$-compatible. If $\overline{g}$ is also invertible, we may replace in the proposition inserter by iso-inserter.
 \end{\prop}

 \begin{proof}
The same remark that we made at the beginning of the proof of Proposition \ref{prop:products} applies here, with respect to \cite[Prop. 2.2]{K2dim}. 
  Given the inserter in $\cc{K}$, $(A,(p,\lambda))$, we have the cone $(q,\mu) := (b \cdot Tp, (\overline{g} \cdot Tp)(c \cdot T\lambda)(\overline{f}^{-1}\cdot Tp))$ and therefore there exists a unique $TA \mr{a} A$ such that $pa = q$ and $\lambda a = \mu$. From here it follows, as in \mbox{\cite[Prop. 2.2]{K2dim},} that $(A,a) \in T$-$Alg_{\omega}^{\Omega}$, that $p$ is strict and that $\lambda$ is an algebra $2$-cell, so that $(p,\lambda)$ is a cone in $T$-$Alg_{\omega}^{\Omega}$.
  
  Now, to prove the universal property, consider a cone $(q,\mu)$ in $T$-$Alg_{\omega}^{\Omega}$, with vertex $D$. Using the universal property in $\K$, we have a unique $D \mr{h} A$ such that $ph=q$, $\lambda h = \mu$. Now, a key observation for the proof (see \cite[Prop. 2.2]{K2dim} for the computations) is that the equation \eqref{eq:alg2cell} that expresses that $\mu$ is an algebra $2$-cell is equivalent to the equation \mbox{$(\lambda h d)(f \overline{q}) = (g \overline{q})(\lambda a Th)$,} which expresses that $\overline{q}$ is a morphism between the cones \mbox{$(p \cdot a \cdot Th,\lambda a Th)$} and $(p \cdot h \cdot d, \lambda h d)$, see Example \ref{ex:inserter}. Then, since $\overline{q}$ is in $\Omega$, and the inserter is $\Omega$-compatible, there is a unique $2$-cell $a \cdot Th \Mr{\overline{h}} hd$ in $\Omega$ such that $p \overline{h} = \overline{q}$. Note that, if the inserter is $\Omega'$-compatible, $\overline{h}$ is in $\Omega'$ if $\overline{q}$ is, giving the last assertion of the 
proposition. The 
rest of the proof is exactly as in \cite{K2dim}.
 \end{proof}

 \begin{\prop} \label{prop:equifiers}
  The forgetful $2$-functor $T$-$Alg_{\omega}^{\Omega} \mr{U_{\omega}^{\Omega}} \cc{K}$ creates equifiers of $2$-cells \mbox{$\alpha,\beta: (f,\overline{f}) \Mr{} (g,\overline{g})$} for which $\overline{f}$ is invertible and the equifier of $(\alpha,\beta)$ in $\K$ is $\Omega$-compatible. The projection of the equifier is strict, and detects $\Omega'$-ness for any family $\Omega'$ such that the equifier of $(\alpha,\beta)$ in $\K$ is also $\Omega'$-compatible. 
   \end{\prop}

 \begin{proof}
 The same remark of the proofs of Propositions \ref{prop:products}, \ref{prop:inserters} applies here, with respect to \cite[Prop. 2.3]{K2dim}. 
  Given the equifier $p$ in $\cc{K}$, since $\overline{f}$ is invertible then  
  $b \cdot Tp$ equifies $\alpha$ and $\beta$ and therefore   
  we have a unique $TA \mr{a} A$ with $pa = b \cdot Tp$. Then, as in \cite[Prop. 2.3]{K2dim}, $(A,a) \in T$-$Alg_\omega^{\Omega}$ and $p$ is strict.
  
  Now, to show that $(p,id)$ is the equifier in $T$-$Alg_\Omega$, let $(q,\overline{q})$ equify $\alpha$ and $\beta$ with $\overline{q} \in \Omega$. By the one-dimensional universal property in $\cc{K}$, there is a unique $D \mr{h} A$ with $ph = q$. Now, since the equifier is $\Omega$-compatible (see Example \ref{ex:equifier}), given the $2$-cell $pa . Th \Mr{\overline{q}} phd$ there
is a unique 2-cell $a. Th \Mr{\overline{h}} hd$ in $\Omega$ with $p\overline{h} = \overline{q}$. 
   Note that, if the equifier is $\Omega'$-compatible, $\overline{h}$ is in $\Omega'$ if $\overline{q}$ is, giving the last assertion of the proposition. The rest of the proof is exactly as in \cite{K2dim}. 
 \end{proof}

\section{Weak limits in the $2$-categories of weak morphisms} \label{sec:weaklimitsofalgebras}
 
 We now give what we consider our main result of this paper, a lifting theorem for $2$-categories of $\omega$-$T$-morphisms and conical $\sigma$-$\omega op$-limits (that is, Gray's cartesian op-quasi-limits, see Remark \ref{rem:oplimits}). Various interesting limit lifting results follow as corollaries.
%for $2$-categories of algebras  
 
 \begin{theorem} \label{teo:main}
 Let $\Sigma$ be a family of arrows of $\A$, and let $\Omega$, $\Omega'$ $\Omega''$ be families of $2$-cells of $\K$. 
We consider the family $\overline{\Omega}$ consisting of those $2$-cells $\alpha$ of $T$-$Alg_{\omega}^{\Omega'}$ such that $\alpha \in \Omega$. 
 Assume $T(\Omega) \subseteq \Omega$.
  Consider a $2$-functor $\A \mr{\overline{F}} T$-$Alg_{\omega}^{\Omega'}$, denote $F = U_{\omega}^{\Omega'} \overline{F}$. 
  If $\overline{F(f)}$ is in $\Omega$ for each $f \in \Sigma$, and $\soopLim{A \in \A}{FA}$ (with respect to $\Sigma$, $\Omega$) exists in $\K$ and is $\Omega'$-compatible, then $\soopLim{A \in \A}{\overline{F}A}$ (with respect to $\Sigma$, $\overline{\Omega}$) exists in $T$-$Alg_{\omega}^{\Omega'}$ and is preserved by $U_{\omega}^{\Omega'}$. 
 We may say in this case, slightly abusing language, that the forgetful $2$-functor $U_{\omega}^{\Omega'}$ creates this type of $\sigma$-$\omega op$-limits. 
  The projections $\pi_A$ of this limit are strict, and jointly detect $\Omega''$-ness for any family $\Omega''$ such that $\soopLim{A \in \A}{FA}$ is also $\Omega''$-compatible.
 \end{theorem}

\begin{proof}
 Denote $L = \soopLim{A \in \A}{FA}$. We define a $\sigma$-$\omega$-cone $\theta = (\theta_A, \theta_f)$ with vertex $TL$, where $\theta_A = a T(\pi_A)$ and  $\theta_f$ is given by the composition
 
 \begin{equation} \label{eq:diagconop}
 \xymatrix{& TL \ar[rd]^{T \pi_B} \ar[dl]_{T \pi_A} \ar@{}[d]|{\stackrel{\Leftarrow}{T\pi_f}} \\
 TFA \ar[rr]_{TFf} \ar[d]_a \ar@{}[drr]|{\Downarrow \; \overline{Ff}} && TFB \ar[d]^b \\
				  A \ar[rr]_{Ff} && B}
\end{equation}

Note that for $f \in \Sigma$, $\pi_f \in \Omega$ then by hypothesis so is $T\pi_f$, and $\overline{Ff} \in \Omega$ also by hypothesis, therefore $\theta_f \in \Omega$ as desired. From the one-dimensional universal property of the limit $\soopLim{A \in \A}{FA}$, we have a unique $TL \mr{l} L$ such that $\pi_A l = \theta_A$ and $\pi_f l = \theta_f$ for every $A,f$. The associativity and unit axioms for $l$, expressing that $(L,l)$ is a $T$-algebra, follow from those of the $FA$ using that the projections $\pi_A$ are jointly monic and the naturality of the unit and the multiplication of $T$. Then the $\pi_A$ are strict morphisms as desired, and the equality $\pi_f l = \theta_f$ is exactly the equation that expresses that the $\pi_f$ are algebra $2$-cells, thus we have a $\sigma$-$\omega$-cone in $T$-$Alg_{\omega}^{\Omega'}$.

To show the one-dimensional universal property in $T$-$Alg_{\omega}^{\Omega'}$, consider another $\sigma$-$\omega$-cone $E \xr{\alg{h_A}} FA$, $\alg{hb} \Mr{h_f} \alg{Ff} \alg{ha}$. We will show that there is a unique $E \xr{\alg{h}} L$ such that $\alg{h_A} = (\pi_A,id) \alg{h}$ for each $A$, that is $\pi_A h = h_A$, $\pi_f h = h_f$ and $\pi_A \overline{h} = \overline{h_A}$. By the universal property in $\K$, there exists a unique $h$ satisfying the first two of these equalities. Now, the equations that express the fact that each $h_f$ is an algebra $2$-cell, namely

$$\vcenter{\xymatrix{ && \\
TE \ar@/^2.7pc/[rr]^{Th_B} \ar[r]^{Th_A} \ar[d]_e & \ar@{}[u]|<<<<{\Downarrow \: Th_f}    TFA \ar[r]^{TFf} \ar[d]_a & TFB \ar[d]^b \\
E \ar@{}[ru]|{\Downarrow \; \overline{h_A}} \ar[r]_{h_A} & FA \ar@{}[ru]|{\Downarrow \; \overline{Ff}} \ar[r]_{Ff} & FB }}
=
\vcenter{\xymatrix{ TE \ar[rr]^{Th_B} \ar@{}[rrd]|{\Downarrow \; \overline{h_B}} \ar[d]_e && TFB \ar[d]^b  \\
E \ar@/_2ex/[rd]_{h_A} \ar[rr]^{h_B} && FB \\
& FA \ar@{}[u]|{\Downarrow \; h_f} \ar@/_2ex/[ru]_{Ff} }}$$

are equivalent to the axiom expressing that $h_* e \mr{\overline{h_A}} \theta_* Th$ is a morphism of $\sigma$-$\omega$-cones, that is a modification. Note that $h_* e = \pi_* h e$, $\theta_* Th = \pi_* l Th$, then by our $\Omega'$-compatibility hypothesis we have a unique $he \Mr{\overline{h}} l Th$ in $\Omega'$ satisfying $\pi_A \overline{h} = \overline{h_A}$ as desired. 
Note that, if the limit is $\Omega''$-compatible, $\overline{h}$ is in $\Omega''$ if each $\overline{h_A}$ is, giving the last assertion of the theorem. 
The coherence conditions with respect to $m,i$ for $\overline{h}$ follow from those of the $\overline{h_A}$ using that, by the two-dimensional universal property of the limit in $\K$, if a pair of $2$-cells is equal after composing with all the $\pi_A$, then they are equal.

Now, for the $2$-dimensional universal property in $T$-$Alg_{\omega}^{\Omega'}$, consider the following situation:

$$\xymatrix{ E \ar@<4ex>[rrd]^{\alg{h_A}} \ar@<0ex>[rrd]_{\alg{g_A}} \ar@{}@<2ex>[rrd]|{\Downarrow \; \beta_A} 
	       \ar@<-0ex>[dd]^{\alg{g}} \ar@<-3ex>[dd]_{\alg{h}} \ar@{}@<-1.5ex>[dd]|{\Rightarrow}  \ar@{}@<-1.5ex>[dd]|<<<<<<<<<{\alpha} \\
	       && FA \\
	       L \ar[rru]_{\pi_A} } \quad\quad\quad
\xymatrix{ E \ar@<4ex>[rd]^{h_A} \ar@<0ex>[rd]_{g_A} \ar@{}@<2ex>[rd]|{\Downarrow \; \beta_A} 
	       \ar@<-0ex>[dd]^{g} \ar@<-3ex>[dd]_{h} \ar@{}@<-1.5ex>[dd]|{\Rightarrow}  \ar@{}@<-1.5ex>[dd]|<<<<<<<<<{\alpha} \\
	       & FA \\
	       L \ar[ru]_{\pi_A} }$$

That is, given a morphism $\beta_A$ between the cones $h,g$ in $T$-$Alg_{\omega}^{\Omega'}$, we have to show that there is a unique algebra $2$-cell $\alpha$ as above. Now, by the $2$-dimensional universal property in $\K$ (use the diagram in the right), there is a unique $2$-cell $\alpha$ in $\K$. It suffices to check that if the $\beta_A$ are algebra $2$-cells then so is $\alpha$. To do this it suffices to compose the two diagrams in \eqref{eq:alg2cell} for $\alpha$ with each $\pi_A$, in order to have the corresponding diagrams for $\beta_A$ using the equalities $\pi_A \overline{g} = \overline{g_A}$, $\pi_A \overline{h} = \overline{h_A}$ and $\pi_A \alpha = \beta_A$.
\end{proof}

\begin{example} \label{ex:inserterfinal}
Let $\alg{f}, \alg{g}$ be a pair of morphisms in $T$-$Alg_\omega^\Omega$, and consider another family $\Omega'$ as in Proposition \ref{prop:inserters}. As it is done in Example \ref{ex:inserter2}, we may write the inserter of $\alg{f}, \alg{g}$ as the $\sigma$-$\omega op$-limit of the diagram $\{a \mrpair{u}{v} b\} \mapsto \{B \mrpair{f}{g} C\}$, with respect to the families $\Sigma = \{u,id_a,id_b\}$ and $\Omega_s$. Then, an application of Theorem \ref{teo:main} to this limit (with $\Omega = \Omega_s$, $\Omega' = \Omega$ and $\Omega'' = \Omega'$) yields the result of Proposition \ref{prop:inserters}, but with the extra hypothesis that $\overline{f}$ is now required to be the identity (this still has \cite[Prop. 4.4]{Llax} as a particular case, but not \cite[Prop 2.2]{K2dim}).

We note that if the inserter is written as a conical $\sigma$-$\omega op$-limit using Proposition \ref{prop:weightedcomoconical} (see again Example \ref{ex:inserter2}), then the obtained result is even weaker, since in this case both $\overline{f}$ and $\overline{g}$ are now required to be identities, in other words only inserters in $T$-$Alg_s$ are obtained in this way. We don't know if a stronger result could be deduced from Theorem \ref{teo:main} using another expression of the inserter as a conical $\sigma$-$\omega op$-limit. 
\end{example}

\begin{remark} 
The reason why \cite[Prop. 2.2]{K2dim} doesn't follow from Theorem \ref{teo:main} using Proposition \ref{prop:weightedcomoconical} is that we are trying to lift a $\sigma$-$\omega$-limit which is taken with respect to $\Omega_s$ to a $2$-category of $\omega$-$T$-morphisms with respect to $\Omega_p$, but $\Omega_s$ is not included in $\Omega_p$, therefore our hypothesis that 
$\overline{F(f)}$ is in $\Omega_s$ for each $f \in \Sigma$ won't necessarily hold.
This obstruction vanishes in the case of Theorem \ref{teo:main} in which $\Omega' \subseteq \Omega$, 
and this allows the use of Proposition \ref{prop:weightedcomoconical} to deduce the lifting of weighted $\sigma$-$\omega$-limits:
\end{remark}

\begin{corollary} \label{coro:corogeneral}
 Let $\Sigma$ be a family of arrows of $\A$, and let $\Omega$, $\Omega'$ $\Omega''$ be families of $2$-cells of $\K$. Assume $T(\Omega) \subseteq \Omega$ and $\Omega' \subseteq \Omega$.
 Then the forgetful $2$-functor $U_{\omega}^{\Omega'}$ creates $\Omega'$-compatible (weighted) $\sigma$-$\omega op$-limits (with respect to $\Sigma$, $\Omega$, in the sense of Theorem \ref{teo:main}). 
  The projections of these limits are strict, and detect $\Omega''$-ness for any family $\Omega''$ such that $\soopLim{A \in \A}{FA}$ is also $\Omega''$-compatible.
\end{corollary}

\begin{proof}
We can write a weighted $\sigma$-$\omega op$-limit as a conical one using Proposition \ref{prop:weightedcomoconical} 
(recall Remark \ref{rem:oplimits}). 
Since $\Omega' \subseteq \Omega$, all the arrows $\alg{g}$ of $T$-$Alg_\omega^{\Omega'}$ satisfy $\overline{g} \in \Omega$, thus the hypothesis of Theorem \ref{teo:main} requiring $\overline{F(f)} \in \Omega$ is immediately satisfied. 
Note that Proposition \ref{prop:weightedcomoconicalcompatible} shows that the conditions of compatibility are preserved by the application of Proposition \ref{prop:weightedcomoconical}, and Remark \ref{rem:correspconos} shows the correspondence between the projections, thus giving the last statement of the corollary.
\end{proof}

\begin{remark}
%$U_{\omega}^{\Omega'}$ creates
 We could actually require in Corollary \ref{coro:corogeneral}, instead of $\Omega' \subseteq \Omega$, a condition on the functor $\A \mr{\overline{F}} T$-$Alg_{\omega}^{\Omega'}$ for $\{W,\overline{F}\}_{\sigma,\omega op}$ to be created by $U_{\omega}^{\Omega'}$, namely that $\overline{F(f)}$ is in $\Omega$ for each $A \mr{f} B$ in $\Sigma$ such that $WA \neq \emptyset$. A more careful application of Proposition \ref{prop:weightedcomoconical}, expliciting the involved constructions, would allow to deduce this result from Theorem \ref{teo:main}.
  Writing this with full detail would lead to results slightly stronger than our Corollaries \ref{coro2}, \ref{coro3}, for example we could deduce the lifting to any $T$-$Alg_{\gamma}$ of all limits which have a strict diagram (cf. \mbox{\cite[Prop. 4.1]{Llax}}). For the sake of simplicity, since we are unsure of how much value such a generalization adds, we have refrained here from doing so.
\end{remark}

\begin{remark} \label{rem:teoop}
We leave to the reader to write in its complete form the dual versions of Theorem \ref{teo:main} and Corollary \ref{coro:corogeneral}, which state that the forgetful $2$-functors \mbox{$T$-$Alg_{co\omega}^{\Omega'} \xr{U_{co\omega}^{\Omega'}} \K$} create the corresponding types of $\sigma$-$\omega$-limits. These dual versions are equivalent to the corresponding results by Remarks \ref{rem:darvueltamorf} and \ref{rem:oplimits} (or, can be seen to hold with dual proofs).

We note that, if $\Omega'$ has only invertible $2$-cells, by the last statement in Remark \ref{rem:darvueltamorf} we may consider the forgetful $2$-functor $T$-$Alg_{\omega}^{\Omega'} \xr{U_{\omega}^{\Omega'}} \K$ instead.
\end{remark}

\begin{remark} \label{rem:hipotesisgratis}
In the situation of Theorem \ref{teo:main}, 
if $\Omega$ is any of the three choices $\Omega_s$, $\Omega_p$, $\Omega_\ell$, then the equality $\overline{\Omega} = \Omega$ makes sense and is easily seen to hold (noting for the $\Omega_p$ case that if an invertible $2$-cell satisfies \eqref{eq:alg2cell} then so does its inverse).
Also, for these choices of $\Omega$, 
the hypothesis \mbox{$T(\Omega) = \Omega$} is immediately satisfied. 
If $\Omega'$ is any of those three choices, then the hypothesis of $\Omega'$-compatibility is immediately satisfied (recall Remark \ref{rem:compatiblegratis} and Proposition \ref{prop:weightedcomoconicalcompatible}). The same holds for $\Omega''$. 
\end{remark}

Keeping Remarks \ref{rem:teoop} and \ref{rem:hipotesisgratis} in mind, the three Corollaries \ref{coro1}, \ref{coro2}, \ref{coro3} below follow immediately from Corollary \ref{coro:corogeneral} as particular cases.
Considering $\Omega = \Omega_\ell$, we have:

\begin{corollary} \label{coro1}
For $\gamma = \ell,p,s$, the forgetful $2$-functor $T$-$Alg_\gamma \mr{U_\gamma} \K$ creates (weighted) oplax limits. In particular,
if $\K$ has such limits then so does $T$-$Alg_\gamma$, and $U_\gamma$ preserves them. The projections of these limits are strict, and detect strictness and pseudoness.
\end{corollary}
 
\begin{remark}
The case $\gamma = \ell$ of Corollary \ref{coro1} is \cite[Theorem 4.8]{Llax}, which is the most general limit lifting result of that article.

Also, considering $\gamma = p$ and recalling Remark \ref{rem:teoop}, we have the lax case of \mbox{\cite[Theorem 2.6]{K2dim}:} the forgetful $2$-functor $T$-$Alg_p \mr{U_p} \K$ creates lax limits. We note that this result is also a consequence of Corollary \ref{coro2} below.
\end{remark}

Considering now $\Omega = \Omega_p$ in Corollary \ref{coro:corogeneral}, we have (recall Example \ref{ex:varioslimits}, item 2):

\begin{corollary} \label{coro2}
For $\gamma = p,s$, the forgetful $2$-functor $T$-$Alg_\gamma \mr{U_\gamma} \K$ creates (weighted) \mbox{$\sigma$-limits.} In particular,
if $\K$ has such limits then so does $T$-$Alg_\gamma$, and $U_\gamma$ preserves them. The projections of these limits are strict, and detect strictness.
\end{corollary}

\begin{remark} \label{rem:casosk2dimysigmalim}
Recall that $\sigma$-limits generalize both pseudo and lax limits. Then the case $\gamma = p$ of Corollary \ref{coro2} generalizes \cite[Theorem 2.6]{K2dim} to arbitrary $\sigma$-limits. 

Considering the examples of \cite[Example 6.6]{K2dim}, from Corollary \ref{coro2} it also follows that \mbox{$\sigma$-limits} are computed pointwise in the $Hom_p$ and $Hom_s$ $2$-categories, which is a fundamental fact for the theory of flat $2$-functors (see \cite[Corollary 2.6.4]{DDS1}, where an independent proof of this fact is given without using monad theory).
\end{remark}

Finally, considering $\Omega = \Omega_s$, we recover the classic result of \mbox{$\Cat$-enriched} monad theory (\cite[Prop. II.4.8 a)]{D}):

\begin{corollary} \label{coro3}
The forgetful $2$-functor $T$-$Alg_s \mr{U_s} \K$ creates (weighted) limits (that is $s$-limits). In particular,
if $\K$ has such limits then so does $T$-$Alg_s$, and $U_s$ preserves them.
\end{corollary}

\bibliographystyle{unsrt}

\end{document}